\documentclass[11pt]{article}

\usepackage{amsmath, amsthm, amssymb, mathtools}
\usepackage{booktabs}
\usepackage[hidelinks, colorlinks=false]{hyperref}
\usepackage{cleveref}
\crefformat{maintheorem}{#2Main Theorem#3}
\Crefformat{maintheorem}{#2Main Theorem#3}
\crefname{proposition}{Proposition}{Propositions}
\Crefname{proposition}{Proposition}{Propositions}
\crefname{lemma}{Lemma}{Lemmas}
\Crefname{lemma}{Lemma}{Lemmas}
\crefname{conjecture}{Conjecture}{Conjectures}
\Crefname{conjecture}{Conjecture}{Conjectures}

\newtheorem{theorem}{Theorem}[section]
\newtheorem*{maintheorem}{Main Theorem}
\newtheorem{lemma}[theorem]{Lemma}
\newtheorem{proposition}[theorem]{Proposition}

\theoremstyle{definition}

\newtheorem{remark}[theorem]{Remark}
\newtheorem{conjecture}[theorem]{Conjecture}

\title{Domination in Johnson graphs \texorpdfstring{$J(n, 3)$}{J(n,3)} for odd \texorpdfstring{$n$}{n}}

\author{%
Seung-ah Lee\thanks{Department of Mathematics, Pusan National University, Busan 46241, Republic of Korea. Email: \texttt{sal0403@pusan.ac.kr}.}
\and
Semin Oh\thanks{Corresponding author. KNU G-LAMP Research Center, KNU Institute of Basic Sciences, Kyungpook National University, Daegu 41566, Republic of Korea. Email: \texttt{semin@knu.ac.kr}.}%
}

\date{\today}

\begin{document}

\maketitle

\begin{abstract}
In 2025 Cornet, Dravec, and Torres determined the domination number $\gamma(J(n, 3))$ of the Johnson graph for every even $n \ge 6$, expressing it as a closed form $\phi_n$ in terms of Fort\textendash{}Hedlund covering numbers, and conjectured the same value for odd $n$. We prove this conjecture: $\gamma(J(n, 3)) = \phi_n$ for every odd $n \ge 7$. A dominating set of $J(n, 3)$ is the same as a lotto design $\mathrm{LD}(n, 3, 3, 2)$, and the lotto numbers $L(n, 3, 3, 2)$, and hence $\gamma(J(n, 3))$, were already determined for every $n$ by Bate (1978) and by Brouwer (1981); a short degree argument of F\"uredi, Sz\'ekely, and Zubor (1996) gives the same lower bound for all $n \not\equiv 9, 10 \pmod{12}$. The contribution of this paper is therefore a new proof of the odd case, carried out within the bijective framework of Cornet, Dravec, and Torres.
\end{abstract}

{\small\noindent\textbf{Keywords:} Johnson graph; domination number; lotto design; Fort\textendash{}Hedlund covering; triangle-free graph; edge-covering by triangles.\par}

\smallskip
\noindent\textbf{MSC 2020:} 05C69; 05C35; 05D05; 05B40.

\section{Introduction}

The \emph{Johnson graph} $J(n, k)$ has vertex set $\binom{[n]}{k}$, with two $k$-subsets adjacent if and only if they share exactly $k - 1$ elements. Johnson graphs appear naturally in the study of constant-weight codes~\cite{BrouwerShearerSloaneSmith1990}, covering and packing designs~\cite{ColbournDinitz2007}, and Erd\H{o}s\textendash{}Ko\textendash{}Rado type extremal set theory~\cite{ErdosKoRado1961}. They form one of the most-studied families of distance-regular graphs~\cite{BrouwerCohenNeumaier1989}.

Many parameters of Johnson graphs have received extensive attention: the independence number via constant-weight codes~\cite{BrouwerShearerSloaneSmith1990,EtzionBitan1996}, the chromatic number by Etzion and Bitan~\cite{EtzionBitan1996}, the automorphism group by Ramras and Donovan~\cite{RamrasDonovan2011}, and the competition number by Kim, Park, and Sano~\cite{KimParkSano2010}. The domination number has received less attention in the graph-theoretic literature. The closely related Kneser graphs $K(n, k)$ use disjointness, rather than large intersection, as the adjacency relation. Their domination has been investigated by Ivan\v{c}o and Zelinka~\cite{IvancoZelinka1993}, by Gorodezky~\cite{Gorodezky2007}, and by \"{O}sterg\aa{}rd, Shao, and Xu~\cite{OstergardShaoXu2014}. The Roman domination number of Johnson graphs was studied by Zec~\cite{Zec2023}. Domination in Johnson graphs is implicit in the lottery problem: Gr\"undlingh~\cite{Grundlingh2004} formulated lottery numbers as domination numbers of \emph{lottery graphs}, whose vertices are the $k$-subsets of an $n$-set, two being adjacent when they share at least a prescribed number of elements, and noted that the Johnson graphs belong to this family. Fabila-Monroy and G\'omez-Galicia~\cite{FabilaMonroyGomezGalicia2025} studied the domination number of token graphs, including the token graphs $F_k(K_n) \cong J(n, k)$ of complete graphs. Cornet, Dravec, and Torres~\cite{CornetDravecTorres2025} gave closed forms for $\gamma(J(n, 2))$ and, for even $n$, for $\gamma(J(n, 3))$; as we explain below, the case $k = 3$ had been settled much earlier in the language of lotto designs.

The case $J(n, 1) = K_n$ is trivial. For $J(n, 2) \cong L(K_n)$, an elementary argument gives $\gamma(J(n, 2)) = \lfloor n/2 \rfloor$ for every $n \ge 4$~\cite[Theorem~6]{CornetDravecTorres2025}. Cornet, Dravec, and Torres~\cite{CornetDravecTorres2025} determined $\gamma(J(n, 3))$ for every even $n \ge 6$ and conjectured the analogous closed form for odd $n$. Both statements are expressed in terms of a residue-class quantity $\phi_n$ obtained from Fort\textendash{}Hedlund covering numbers; we give the precise definition of $\phi_n$ at~\eqref{eq:phi-odd} in~\Cref{sec:prelim}. We first recall their result and conjecture.

\begin{theorem}[\cite{CornetDravecTorres2025}]\label{thm:CDT-even}
For every even integer $n \ge 6$,
\[
\gamma(J(n, 3)) \;=\; \phi_n.
\]
\end{theorem}

\begin{conjecture}[odd case of~{\cite[Conjecture~12]{CornetDravecTorres2025}}]\label[conjecture]{conj:CDT-odd}
For every odd integer $n \ge 7$,
\[
\gamma(J(n, 3)) \;=\; \phi_n.
\]
\end{conjecture}

We give a proof of the following known identity within the local framework of Cornet, Dravec, and Torres~\cite{CornetDravecTorres2025}.

\begin{maintheorem}\label[maintheorem]{thm:main}
For every odd integer $n \ge 7$,
\[
\gamma(J(n, 3)) \;=\; \phi_n.
\]
\end{maintheorem}

Combined with~\Cref{thm:CDT-even}, this gives $\gamma(J(n, 3)) = \phi_n$ for every $n \ge 6$. (One step in the proof of~\Cref{thm:CDT-even} given in~\cite{CornetDravecTorres2025} needs an additional argument; see the concluding remarks. The statement itself also follows from the results of Bate and Brouwer recalled below.)

\paragraph{Relation to lotto designs.} A \emph{lotto design} $\mathrm{LD}(n, k, p, t)$ is a family of $k$-subsets of $[n]$ such that every $p$-subset of $[n]$ shares at least $t$ elements with some member of the family; the \emph{lotto number} $L(n, k, p, t)$ is the minimum size of such a family. A family $D \subseteq \binom{[n]}{3}$ dominates $J(n, 3)$ if and only if every $3$-subset of $[n]$ shares at least two elements with some member of $D$, that is, if and only if $D$ is an $\mathrm{LD}(n, 3, 3, 2)$. Hence
\[
\gamma(J(n, 3)) \;=\; L(n, 3, 3, 2).
\]
These lotto numbers were determined for every $n$ by Bate in his thesis~\cite[Chapter~4, Table~4.2.2]{Bate1978}\footnote{The closed form stated in~\cite[Theorem~4.3.1]{Bate1978} has an extra $+1$ for $n \equiv 11 \pmod{12}$; the values in~\cite[Table~4.2.2]{Bate1978}, from which that theorem is derived, are correct and agree with~$\phi_n$. Gr\"undlingh~\cite[p.~7, footnote~5]{Grundlingh2004} noted the same discrepancy in~\cite{BateStanton}.} and, independently, by Brouwer~\cite[Theorem~4]{Brouwer1981}; the formulas were announced in~\cite{BrouwerVoorhoeve1979}. Brouwer's formula agrees with $\phi_n$ for every $n \ge 6$. Thus~\Cref{thm:CDT-even} and~\Cref{thm:main} were both known before~\Cref{conj:CDT-odd} was posed.

Bate~\cite[Sections~4.2--4.3]{Bate1978} already gives a short proof of $\gamma(J(n,3))\ge\phi_n$ for every $n\ge6$. In the language of~\Cref{prop:bijection}, his argument is as follows. Let $G$ be a triangular graph of order $n$ with $\alpha(G)\le2$, so that $\overline G$ is triangle-free. If $\overline G$ is bipartite, then $V(G)$ is the union of two cliques. When both parts have size at least $2$, the triangle-replacement argument in~\Cref{lem:CDT9} gives $\rho_\Delta(G)\ge\phi_n$, including a repartition if one part has size $2$. If a part has size at most $1$, then $|E(G)|\ge\binom{n-1}{2}$, so $\rho_\Delta(G)\ge(n-1)(n-2)/6\ge\phi_n$ by~\eqref{eq:phi-even}--\eqref{eq:phi-odd}. If $\overline G$ is non-bipartite, then~\Cref{lem:nonbip-trianglefree} gives $|E(G)|\ge\binom{n}{2}-\lfloor(n-1)^2/4\rfloor-1$, and $\rho_\Delta(G)\ge\lceil|E(G)|/3\rceil\ge\phi_n$. Brouwer instead proves an extension of Tur\'an's theorem to multigraphs with even degrees and independence number at most~$2$.

F\"uredi, Sz\'ekely, and Zubor~\cite[Lemma~3 and Corollary~2(e)]{FurediSzekelyZubor1996} give a lower bound that avoids the bipartite/non-bipartite split altogether. By a theorem of Erd\H{o}s (see~\cite[(9)]{FurediSzekelyZubor1996}), a graph with independence number at most~$2$ on a vertex set $V$ admits a partition $V = A_1 \sqcup A_2$ such that every $x \in A_i$ has degree at least $|A_i| - 1$. Applied to the underlying graph of the pair multigraph of an $\mathrm{LD}(n, 3, 3, 2)$, in which every degree is even, this yields
\[
L(n, 3, 3, 2) \;\ge\; \min_{0 \le a \le n} \left\lceil \tfrac13 \Bigl( a \bigl\lceil \tfrac{a-1}{2} \bigr\rceil + (n-a) \bigl\lceil \tfrac{n-a-1}{2} \bigr\rceil \Bigr) \right\rceil,
\]
and they observe that the right-hand side equals Brouwer's value whenever $n \not\equiv 9, 10 \pmod{12}$. For $n \equiv 9, 10 \pmod{12}$ it is exactly one less, and the remaining unit is supplied by the arguments of Bate and Brouwer above.

Our proof of the odd case follows a different route: it carries out the local strategy of the proof of~\Cref{thm:CDT-even} in~\cite{CornetDravecTorres2025}, which splits along a non-adjacent pair $u,v$, and supplies a self-contained proof of the odd-order version of their Lemma~10, whose proof was omitted there.

The matching upper bound is provided by an explicit 2-clique construction. For the lower bound we fix a non-adjacent pair $u, v$. Either $G$ reduces to a partition of $V(G)$, or of $V(G)$ minus one vertex, into two cliques (Case~A), or we bound $|E(G)|$ from below using the bound for triangle-free non-bipartite graphs (\Cref{lem:nonbip-trianglefree}) together with a counting argument on vertices with no neighbours across a bipartition (Case~B).

The remainder of the paper is organised as follows.~\Cref{sec:prelim} collects notation, recalls the bijective framework of~\cite{CornetDravecTorres2025}, gives the Fort\textendash{}Hedlund formula and the closed form of $\phi_n$, and exhibits the explicit $2$-clique construction realising the upper bound.~\Cref{sec:proof} establishes the matching lower bound and combines it with the upper bound to prove~\Cref{thm:main}. We leave open the extremal characterisation of the $2$-shadow graphs at which the minimum in~\Cref{prop:bijection} is attained, returning to this point in the concluding remarks.

\section{Preliminaries}\label{sec:prelim}

Throughout, all graphs are finite, undirected, simple, and loopless. For a graph $G = (V(G), E(G))$, an \emph{independent set} is a vertex subset spanning no edges, and the \emph{independence number} $\alpha(G)$ is the maximum size of such a set. A \emph{dominating set} of $G$ is a vertex subset $D \subseteq V(G)$ such that every vertex of $G$ either lies in $D$ or has a neighbour in $D$, and the \emph{domination number} $\gamma(G)$ is the minimum size of a dominating set; we refer to~\cite{HaynesHedetniemiHenningCore2023} for general background on domination. We write $N(v)$ for the open neighbourhood of $v$, $G - S$ for the induced subgraph on $V(G) \setminus S$, $\overline{G}$ for the complement, and $G \sqcup H$ for the disjoint union. A graph is \emph{triangular} if every edge lies in some triangle.\footnote{We follow the terminology of~\cite{CornetDravecTorres2025}; this notion is unrelated to the alternative use of \emph{triangulated} for chordal graphs.}

For positive integers $n \ge k \ge 1$, the \emph{Johnson graph} $J(n, k)$ has vertex set $\binom{[n]}{k}$, with two vertices adjacent if and only if their intersection has size $k - 1$; thus $J(n, k)$ is $k(n - k)$-regular and has $\binom{n}{k}$ vertices.

The bijective framework underlying our argument is due to Cornet, Dravec, and Torres~\cite{CornetDravecTorres2025}. For a graph $G$, an \emph{edge-covering by triangles} of $G$ is a set $T$ of triangles of $G$ such that every edge of $G$ lies in some triangle of $T$; such a covering exists if and only if $G$ is triangular. Write $\rho_{\Delta}(G)$ for the minimum size of such a $T$, with the convention $\rho_{\Delta}(G) := \infty$ if $G$ is not triangular. For a family $D \subseteq \binom{[n]}{3}$, the \emph{$2$-shadow graph} of $D$ is the graph on vertex set $[n]$ whose edges are exactly the pairs $\{x, y\} \subseteq [n]$ contained in at least one triple of $D$. The following proposition, which we recall without proof, identifies dominating sets of $J(n, 3)$ with edge-coverings by triangles of triangular graphs $G$ on $[n]$ with $\alpha(G) \le 2$.

\begin{proposition}[\cite{CornetDravecTorres2025}]\label[proposition]{prop:bijection}
A subset $D \subseteq V(J(n, 3))$ is a dominating set of $J(n, 3)$ if and only if $D$ is an edge-covering by triangles of some triangular graph $G$ on vertex set $[n]$ with $\alpha(G) \le 2$. Consequently, $\gamma(J(n,3))=t_n$, where
\[
t_n \;:=\; \min \bigl\{ \rho_{\Delta}(G) : G \text{ is triangular},\ |V(G)| = n,\ \alpha(G) \le 2 \bigr\}.
\]
\end{proposition}

For cliques, the relevant covering numbers were obtained by Fort and Hedlund~\cite{FortHedlund1958}. For integers $m \ge 3$, define
\[
|C_m| \;:=\; \rho_{\Delta}(K_m).
\]
The Fort\textendash{}Hedlund closed form is
\begin{equation}\label{eq:FH}
|C_m| \;=\;
\begin{cases}
m^2/6                & \text{if } m \equiv 0 \pmod{6}, \\
m(m - 1)/6           & \text{if } m \equiv 1, 3 \pmod{6}, \\
(m^2 + 2)/6          & \text{if } m \equiv 2, 4 \pmod{6}, \\
(m^2 - m + 4)/6      & \text{if } m \equiv 5 \pmod{6}.
\end{cases}
\end{equation}
Following~\cite{CornetDravecTorres2025}, for $n \ge 6$ define
\[
\phi_n \;:=\; \min \bigl\{ |C_{m_1}| + |C_{m_2}| : m_1 + m_2 = n,\ m_1, m_2 \ge 3 \bigr\}.
\]

\begin{proposition}[\cite{CornetDravecTorres2025}]\label[proposition]{prop:phi-upper}
For every $n \ge 6$, $\gamma(J(n, 3)) \le \phi_n$.
\end{proposition}

An explicit construction realising~\Cref{prop:phi-upper} is as follows. Given any split $n = m_1 + m_2$ with $m_1, m_2 \ge 3$ achieving $|C_{m_1}| + |C_{m_2}| = \phi_n$, partition $[n] = A_1 \sqcup A_2$ with $|A_i| = m_i$ and take minimum edge-coverings by triangles $\mathcal{D}_i \subseteq \binom{A_i}{3}$ of $K_{m_i}$ on $A_i$, explicitly given by the Fort\textendash{}Hedlund residue-class recipe~\cite[\S 2]{FortHedlund1958}. The union $\mathcal{D}_1 \sqcup \mathcal{D}_2$ has size $\phi_n$ and is a dominating set of $J(n, 3)$ by~\Cref{prop:bijection}, since its $2$-shadow $K_{m_1} \sqcup K_{m_2}$ is triangular with $\alpha = 2$.

The quantity $\phi_n$ admits parity-dependent closed forms. A residue computation based on the Fort\textendash{}Hedlund formula yields the following closed forms, recorded in~\cite[p.~227]{CornetDravecTorres2025}. For even $n \ge 6$,
\begin{equation}\label{eq:phi-even}
\phi_n \;=\; \frac{n^2 - 2n + a_n}{12}, \qquad
a_n \;=\;
\begin{cases}
0  & \text{if } n \equiv 2, 6 \pmod{12}, \\
4  & \text{if } n \equiv 4 \pmod{12}, \\
12 & \text{if } n \equiv 0, 8 \pmod{12}, \\
16 & \text{if } n \equiv 10 \pmod{12}.
\end{cases}
\end{equation}
For odd $n \ge 7$,
\begin{equation}\label{eq:phi-odd}
\phi_n \;=\; \frac{n^2 - n + b_n}{12}, \qquad
b_n \;=\;
\begin{cases}
0  & \text{if } n \equiv 1 \pmod{12}, \\
4  & \text{if } n \equiv 5 \pmod{12}, \\
6  & \text{if } n \equiv 3, 7 \pmod{12}, \\
10 & \text{if } n \equiv 11 \pmod{12}, \\
12 & \text{if } n \equiv 9 \pmod{12}.
\end{cases}
\end{equation}
In particular, $b_n \le 12$ for every odd $n$.

The following standalone bound for the clique covering number will be used in~\Cref{sec:proof}.

\begin{lemma}\label[lemma]{lem:clique-case}
For every odd integer $n \ge 7$, $|C_n| \ge \phi_n$.
\end{lemma}

\begin{proof}
Choose the split $(m_1, m_2) = (n - 3, 3)$ in the definition of $\phi_n$, valid since $n - 3 \ge 4$ for $n \ge 7$; this gives $\phi_n \le |C_{n - 3}| + |C_3| = |C_{n - 3}| + 1$. The Fort\textendash{}Hedlund formula~\eqref{eq:FH} yields $|C_n| \ge |C_{n - 3}| + 1$ by a residue-by-residue substitution: the difference $|C_n| - |C_{n - 3}|$ equals $(5n - 11)/6$, $(5n - 9)/6$, or $(5n - 7)/6$ according as $n \equiv 1,\,3,\,5 \pmod 6$, each of which is $\ge 1$ for $n \ge 7$. Hence $|C_n| \ge |C_{n-3}| + 1 \ge \phi_n$.
\end{proof}

The following classical bound on triangle-free non-bipartite graphs will also be used in~\Cref{sec:proof}.

\begin{lemma}[Erd\H{o}s; see~{\cite[Exercise~7.3.3]{BondyMurty1976}}]\label[lemma]{lem:nonbip-trianglefree}
For every $K_3$-free non-bipartite graph $G$ on $n$ vertices,
\[
|E(G)| \;\le\; \left\lfloor \frac{(n - 1)^2}{4} \right\rfloor + 1.
\]
\end{lemma}

The bound follows by applying Mantel's theorem~\cite{Mantel1907} outside a shortest odd cycle. If $C$ is a shortest odd cycle, of length $2k+1\ge5$, then $C$ has no chord, and each vertex outside $C$ has at most $k$ neighbours on $C$. Hence
\begin{align*}
|E(G)| &\;\le\; \frac{(n-2k-1)^2}{4} + (2k+1) + k(n-2k-1) \\
&\;=\; \frac{(n-1)^2}{4} + 2 - (k-1)^2 \;\le\; \frac{(n-1)^2}{4} + 1.
\end{align*}

\section{Proof of the main theorem}\label{sec:proof}

By~\Cref{prop:bijection} and~\Cref{prop:phi-upper}, it suffices to prove $\rho_{\Delta}(G) \ge \phi_n$ for every triangular graph $G$ of order $n$ with $\alpha(G) \le 2$. If $\alpha(G) = 1$, then $G = K_n$ and~\Cref{lem:clique-case} applies; we may assume $\alpha(G) = 2$.

We use two general inequalities throughout. Since each triangle of $G$ covers at most three edges, an elementary double count gives
\begin{equation}\label{eq:3D-EG}
3 \, \rho_{\Delta}(G) \;\ge\; |E(G)|.
\end{equation}
For any non-adjacent pair $u, v \in V(G)$, write $H := G - \{u, v\}$ and $X_{uv} := N_G(u) \cap N_G(v)$. Since $uv \notin E(G)$, every edge of $G$ either lies inside $H$ or has $u$ or $v$ as an endpoint with the other endpoint in $V(H)$. Therefore
\[
|E(G)| \;=\; |E(H)| + |N_G(u) \cap V(H)| + |N_G(v) \cap V(H)|.
\]
Since $\alpha(G) \le 2$, every vertex of $V(H)$ lies in $N_G(u) \cup N_G(v)$: otherwise such a vertex would, together with $u$ and $v$, form an independent triple. Hence $|N_G(u) \cap V(H)| + |N_G(v) \cap V(H)| = (n - 2) + |X_{uv}|$, and combining,
\begin{equation}\label{eq:E-decomp}
|E(G)| \;=\; |E(H)| + (n - 2) + |X_{uv}|.
\end{equation}

The argument splits along the dichotomy
\begin{itemize}
\item \emph{Case~A}: there exists a non-adjacent pair $u, v \in V(G)$ with $|X_{uv}| \le 1$, or with $G - \{u, v\}$ disconnected;
\item \emph{Case~B}: every non-adjacent pair $u, v \in V(G)$ satisfies $|X_{uv}| \ge 2$ and $G - \{u, v\}$ is connected,
\end{itemize}
which is exhaustive, since Case~B is the negation of Case~A. The two cases play asymmetric roles in the proof.

Case~A reduces $G$ to a $2$-clique partition (of $V(G)$ or of $V(G) - \{w\}$), then applies structural lemmas of~\cite{CornetDravecTorres2025}. In Case~B, when the complement of $G - \{u, v\}$ is bipartite, a coupling argument propagates an edge-count bound from cross-isolated vertices to $\{u, v\}$.

\subsection{Case A: Reduction to 2-clique partitions}\label{sec:caseA}

The following two lemmas of~\cite{CornetDravecTorres2025} give $\rho_{\Delta}(G) \ge \phi_n$ when $V(G)$, or $V(G) - \{w\}$ for some $w$, admits a partition into two cliques.

\begin{lemma}[adapted from~{\cite[Lemma~9]{CornetDravecTorres2025}}]\label[lemma]{lem:CDT9}
Let $G$ be a triangular graph of order $n \ge 6$ with $\alpha(G) \le 2$. If $V(G)$ admits a partition into two cliques, each of size at least $2$, then $\rho_{\Delta}(G) \ge \phi_n$.
\end{lemma}

Lemma~9 of~\cite{CornetDravecTorres2025} assumes $\alpha(G)=2$ and $\rho_\Delta(G)=t_n$, and concludes $t_n=\phi_n$. The following triangle-replacement argument from its proof establishes the stronger form stated here without the global optimality assumption.

\begin{proof}
Let $D$ be any edge-covering by triangles of $G$, and write $V(G)=A\sqcup B$ for the given clique partition. If $|A|=2$, a triangle of $D$ containing the edge of $A$ has its third vertex $y$ in $B$. Replace the partition by $A\cup\{y\}$ and $B\setminus\{y\}$. These are cliques of sizes $3$ and $n-3\ge3$. The same argument applies if $|B|=2$. Thus we may assume both parts have size at least $3$.

Keep every triangle of $D$ lying in one part. Replace each remaining triangle by a triangle inside the part containing two of its vertices, preserving the edge between those two vertices. Such a replacement exists because the part has size at least $3$. Every edge internal to either clique remains covered, and the resulting set $D'$ has $|D'|\le|D|$. Consequently
\[
|D|\;\ge\;|D'|\;\ge\;|C_{|A|}|+|C_{|B|}|\;\ge\;\phi_n.
\]
Taking the minimum over $D$ proves the claim.
\end{proof}

\begin{samepage}
\begin{lemma}[\cite{CornetDravecTorres2025}, Lemma~10]\label[lemma]{lem:CDT10}
Let $n \ge 7$ be odd and let $G$ be a triangular graph of order $n$ with $\alpha(G) \le 2$. If there exists $w \in V(G)$ such that $V(G) - \{w\}$ admits a partition into two cliques $A \sqcup B$ (one part may be empty), then $\rho_{\Delta}(G) \ge \phi_n$.
\end{lemma}
\end{samepage}

\begin{remark}\label{rmk:lemma10-form}
The statement of Lemma~10 in~\cite{CornetDravecTorres2025} additionally assumes $\rho_{\Delta}(G) = t_n$ and concludes $t_n = \phi_n$; the form here, which bounds $\rho_{\Delta}(G)$ for each $G$ in the class, implies that statement by~\Cref{prop:phi-upper}, but not conversely, and it is the one we apply directly in the case analysis. We prove it only for odd $n \ge 7$, which is all we need. The proof is omitted in~\cite{CornetDravecTorres2025} ``due to length restrictions''; we supply a self-contained proof below.
\end{remark}

The proof of~\Cref{lem:CDT10} uses the following lemma, which handles the case when one part of the 2-clique partition is a singleton.

\begin{lemma}\label[lemma]{lem:singleton-degenerate}
Let $G$ be a triangular graph of order $n$ (odd, $n \ge 7$) with $\alpha(G) \le 2$. Suppose $V(G) = X \sqcup \{v\}$, where $X$ is a clique of size $n - 1$. Then $\rho_{\Delta}(G) \ge \phi_n$.
\end{lemma}

\begin{proof}
Since $X$ is a clique of size $n - 1$, the induced subgraph $G[X]$ equals $K_{n-1}$. Let $d := |N_G(v) \cap X|$.

\textit{Case~(1): $d = 0$.} Then $G = K_1 \sqcup K_{n-1}$ and $\rho_{\Delta}(G) = |C_{n-1}|$. Since $n - 1 \ge 6$ is even, its residue modulo~$6$ lies in $\{0, 2, 4\}$, and the Fort\textendash{}Hedlund formula~\eqref{eq:FH} gives $|C_{n-1}| \ge (n-1)^2/6$ in each case. Combined with $\phi_n \le (n^2 - n + 12)/12$ (since $b_n \le 12$ by~\eqref{eq:phi-odd}), it suffices to verify $2(n - 1)^2 \ge n^2 - n + 12$, which reduces to $(n - 5)(n + 2) \ge 0$, valid for $n \ge 5$. Hence $|C_{n-1}| \ge \phi_n$.

\textit{Case~(2): $d \ge 1$.} Pick $u \in N_G(v) \cap X$. The triangular hypothesis forces the edge $\{v, u\}$ to lie in some triangle $\{v, u, x\}$ of $G$. Since $V(G) - \{v\} = X$, we have $x \in X$. The clique structure on $X$ gives $x \sim u$ automatically, while the triangle requires $x \sim v$, i.e., $x \in N_G(v) \cap X$. Hence $|N_G(v) \cap X| \ge 2$, so $d \ge 2$, and $|E(G)| = \binom{n-1}{2} + d \ge \binom{n-1}{2} + 2 = (n^2 - 3n + 6)/2$. By~\eqref{eq:3D-EG},
\[
\rho_{\Delta}(G) \;\ge\; \frac{|E(G)|}{3} \;\ge\; \frac{n^2 - 3n + 6}{6}.
\]
The inequality $(n^2 - 3n + 6)/6 \ge \phi_n$ is equivalent to $n^2 - 5n + 12 \ge b_n$; at $n = 7$ the left side equals $26$, and since $b_n \le 12$ by~\eqref{eq:phi-odd} and $n^2 - 5n + 12$ is monotone increasing for $n \ge 3$, the inequality holds for every odd $n \ge 7$. Hence $\rho_{\Delta}(G) \ge \phi_n$.
\end{proof}

\begin{proof}[Proof of~\Cref{lem:CDT10}]
Set $A' := N_G(w) \cap A$, $\overline A := A \setminus A'$, $a_0 := |\overline A|$, and analogously $B', \overline B$, $b_0 := |\overline B|$. Note that $|A| + |B| = n - 1$.

\textit{Case~(1): $a_0 = 0$ or $b_0 = 0$.} By symmetry, assume $a_0 = 0$, i.e., $w$ is adjacent to every vertex of $A$. We further split on whether $A$ or $B$ is empty.

\textit{Sub-case~(1.1): $A = \emptyset$ or $B = \emptyset$.} Then $V(G) - \{w\}$ is a single clique $C$ of size $n - 1$. Let $\overline C := C \setminus N_G(w)$. If $\overline C = \emptyset$, then $G = K_n$ and~\Cref{lem:clique-case} gives $\rho_{\Delta}(K_n) = |C_n| \ge \phi_n$. Otherwise $V(G) = (\{w\} \sqcup (C \cap N_G(w))) \sqcup \overline C$ is a 2-clique partition: $\{w\} \sqcup (C \cap N_G(w))$ is a clique since $C$ is, and $\overline C \subseteq C$ is a subclique. If both parts have size $\ge 2$,~\Cref{lem:CDT9} applies; if one part is a singleton,~\Cref{lem:singleton-degenerate} applies. Either way, $\rho_{\Delta}(G) \ge \phi_n$.

\textit{Sub-case~(1.2): $A \ne \emptyset$ and $B \ne \emptyset$.} Since $a_0 = 0$, $A \sqcup \{w\}$ is a clique, and $V(G) = (A \sqcup \{w\}) \sqcup B$ is a 2-clique partition. If $|B| \ge 2$,~\Cref{lem:CDT9} applies. If $|B| = 1$, then $X := A \sqcup \{w\}$ is a clique of size $n - 1$, so~\Cref{lem:singleton-degenerate} applies with $v$ being the unique element of $B$. In either case, $\rho_{\Delta}(G) \ge \phi_n$.

\textit{Case~(2): $a_0 \ge 1$ and $b_0 \ge 1$.} We use a direct edge count. First, we claim that
\begin{equation}\label{eq:bar-cross}
\bigl\{ \{a, b\} : a \in \overline A,\ b \in \overline B \bigr\} \;\subseteq\; E(G).
\end{equation}
Indeed, fix $a \in \overline A$ and $b \in \overline B$. By definition of $\overline A$ and $\overline B$, neither $\{w, a\}$ nor $\{w, b\}$ is an edge of $G$. Since $\alpha(G) \le 2$, the triple $\{w, a, b\}$ cannot be independent, hence $\{a, b\} \in E(G)$.

Counting edges of $G$ in disjoint subsets according to the partition $V(G) = A \sqcup B \sqcup \{w\}$: the cliques $K_A, K_B$ give $\binom{|A|}{2} + \binom{|B|}{2}$ edges, the cross-pairs $\overline A \times \overline B$ give $a_0 b_0$ edges by~\eqref{eq:bar-cross}, and $w$ contributes $|A| - a_0$ edges to $A'$ and $|B| - b_0$ edges to $B'$.
Hence
\[
|E(G)| \;\ge\; \binom{|A|}{2} + \binom{|B|}{2} + a_0 b_0 + (|A| - a_0) + (|B| - b_0).
\]
Substituting $|A| + |B| = n - 1$ and using $a_0 b_0 - a_0 - b_0 = (a_0 - 1)(b_0 - 1) - 1 \ge -1$,
\[
|E(G)| \;\ge\; \binom{|A|}{2} + \binom{|B|}{2} + n - 2.
\]
By convexity, $\binom{|A|}{2} + \binom{|B|}{2}$ on $|A| + |B| = n - 1$ is minimised at $|A| = |B| = (n - 1)/2$ (an integer since $n$ is odd), with value $(n - 1)(n - 3)/4$. Hence
\[
|E(G)| \;\ge\; \frac{(n - 1)(n - 3)}{4} + n - 2 \;=\; \frac{n^2 - 5}{4}.
\]
By~\eqref{eq:3D-EG}, $\rho_\Delta(G) \ge |E(G)|/3 \ge (n^2 - 5)/12$, and since $\rho_\Delta(G)$ is an integer,
\[
\rho_\Delta(G) \;\ge\; \left\lceil \frac{n^2 - 5}{12} \right\rceil.
\]
Since $\phi_n = (n^2 - n + b_n)/12$ is an integer by~\eqref{eq:phi-odd}, the inequality $\lceil (n^2 - 5)/12 \rceil \ge \phi_n$ is equivalent to $(n^2 - 5)/12 > \phi_n - 1$, i.e., $n > b_n - 7$. This holds since $b_n \le 12$ and $n \ge 7$, completing Case~(2).
\end{proof}

The next lemma reduces the case when $G - \{u, v\}$ is disconnected to a partition of the type covered by~\Cref{lem:CDT9} or~\Cref{lem:CDT10}.

\begin{lemma}\label[lemma]{lem:disconnection-2clique}
Let $G$ be a triangular graph with $\alpha(G) \le 2$, and let $u, v \in V(G)$ be non-adjacent. If $G - \{u, v\}$ is disconnected, then either
\begin{itemize}
\item[\textnormal{(i)}] $V(G)$ admits a partition into two cliques, each of size at least $2$, or
\item[\textnormal{(ii)}] there exists $w \in \{u, v\}$ such that $V(G) - \{w\}$ admits a partition into two cliques.
\end{itemize}
\end{lemma}

\begin{proof}
Since $G - \{u, v\}$ is disconnected, it has at least two components. If it had three or more, picking one vertex from each of three components would yield an independent triple in $G$, contradicting $\alpha(G) \le 2$. Hence $G - \{u, v\}$ has exactly two components, which we denote $C_1$ and $C_2$.

Each $C_i$ is a clique of $G$, since a non-edge in $C_i$ together with any vertex of the other component would form an independent triple of $G$. Moreover, if $u$ had non-neighbours $x \in C_1$ and $y \in C_2$, then $\{u, x, y\}$ would be independent in $G$; hence $C_i \subseteq N_G(u)$ for some $i \in \{1, 2\}$, and the same holds for $v$.

\textit{Case~(1): $u, v$ share a fully-adjacent component}, i.e., $C_i \subseteq N_G(u) \cap N_G(v)$ for some $i$. Without loss of generality, $i = 1$. Then $\{u\} \sqcup C_1$ is a clique, and $V(G) - \{v\} = (\{u\} \sqcup C_1) \sqcup C_2$ is a 2-clique partition, yielding~(ii) with $w = v$.

\textit{Case~(2): no common fully-adjacent component.} The sets $\{i : C_i \subseteq N_G(u)\}$ and $\{j : C_j \subseteq N_G(v)\}$ are disjoint non-empty subsets of $\{1, 2\}$, so without loss of generality $C_1 \subseteq N_G(u)$ and $C_2 \subseteq N_G(v)$. Then $V(G) = (\{u\} \sqcup C_1) \sqcup (\{v\} \sqcup C_2)$ is a 2-clique partition with both parts of size at least $2$, yielding~(i).
\end{proof}

The following lemma resolves Case~A.

\begin{lemma}\label[lemma]{lem:case-a-general}
Let $n \ge 7$ be odd and let $G$ be a triangular graph of order $n$ with $\alpha(G) \le 2$. Suppose there exists a non-adjacent pair $u, v \in V(G)$ with $|X_{uv}| \le 1$ or with $G - \{u, v\}$ disconnected. Then $\rho_{\Delta}(G) \ge \phi_n$.
\end{lemma}

\begin{proof}
By assumption there is a non-adjacent pair $u, v \in V(G)$ with $|X_{uv}| \le 1$ or with $H := G - \{u, v\}$ disconnected. Two cases arise.

\textit{Case~(1): $H$ is disconnected.}~\Cref{lem:disconnection-2clique} gives either conclusion~(i), in which case~\Cref{lem:CDT9} yields $\rho_{\Delta}(G) \ge \phi_n$, or conclusion~(ii), in which case~\Cref{lem:CDT10} yields the same.

\textit{Case~(2): $H$ is connected and $|X_{uv}| \le 1$.} Set $A_u := (N_G(u) \cap V(H)) \setminus N_G(v)$ and $A_v := (N_G(v) \cap V(H)) \setminus N_G(u)$. Since $\alpha(G) \le 2$, any vertex of $V(H)$ outside $N_G(u) \cup N_G(v)$ would form an independent triple with $u, v$; hence $V(H) = A_u \sqcup A_v \sqcup X_{uv}$. Since $v$ is non-adjacent to every vertex of $A_u$, a non-edge in $A_u$ together with $v$ would form an independent triple; hence $A_u$ is a clique, and symmetrically $A_v$ is a clique.

\textit{Sub-case~(2.1): $X_{uv} = \emptyset$.} Then $V(G) - \{u\} = A_u \sqcup (A_v \sqcup \{v\})$ partitions into two cliques and~\Cref{lem:CDT10} with $w = u$ applies.

\textit{Sub-case~(2.2): $X_{uv} = \{x\}$.} Then $V(G) - \{x\} = (A_u \sqcup \{u\}) \sqcup (A_v \sqcup \{v\})$ partitions into two cliques and~\Cref{lem:CDT10} with $w = x$ applies.
\end{proof}

\subsection{Case B: Coupling argument}\label{sec:caseB}

This case uses~\Cref{lem:nonbip-trianglefree} when $\overline H$ is non-bipartite, and a coupling argument on cross-isolated vertices when $\overline H$ is bipartite.

\begin{lemma}\label[lemma]{lem:case-b-general}
Let $n \ge 7$ be odd, and let $G$ be a triangular graph of order $n$ with $\alpha(G) = 2$. Suppose every non-adjacent pair $u, v \in V(G)$ satisfies $|X_{uv}| \ge 2$ and $G - \{u, v\}$ is connected. Then
\[
|E(G)| \;\ge\; \frac{n^2 - 1}{4},
\]
and consequently
\[
\rho_{\Delta}(G) \;\ge\; \left\lceil \frac{n^2 - 1}{12} \right\rceil \;\ge\; \phi_n.
\]
\end{lemma}

\begin{proof}
Fix any non-adjacent pair $u, v \in V(G)$, which exists since $\alpha(G) = 2$, and set $H:=G-\{u,v\}$. The complement $\overline{H}$ is $K_3$-free, since an independent triple in $H$ would also be independent in $G$, contradicting $\alpha(G) \le 2$. Combining~\eqref{eq:3D-EG} and~\eqref{eq:E-decomp},
\begin{equation}\label{eq:caseB-decomp}
3 \, \rho_{\Delta}(G) \;\ge\; |E(G)| \;=\; |E(H)| + (n - 2) + |X_{uv}|.
\end{equation}
Two cases arise on the bipartiteness of $\overline{H}$.

\textit{Case~(1): $\overline{H}$ is non-bipartite.} Applying~\Cref{lem:nonbip-trianglefree} to $\overline{H}$ on $n - 2$ vertices, and dropping the floor (since $(n - 3)^2/4$ is an integer for $n - 3$ even), gives
\[
|E(\overline{H})| \;\le\; \frac{(n - 3)^2}{4} + 1.
\]
Therefore
\begin{align*}
|E(H)| \;=\; \binom{n - 2}{2} - |E(\overline{H})|
       &\;\ge\; \frac{(n - 2)(n - 3)}{2} - \frac{(n - 3)^2}{4} - 1 \\
       &\;=\; \frac{(n - 3)(n - 1)}{4} - 1.
\end{align*}
Substituting into~\eqref{eq:caseB-decomp} and using $|X_{uv}| \ge 2$,
\[
|E(G)| \;\ge\; \frac{(n - 3)(n - 1)}{4} - 1 + (n - 2) + 2 \;=\; \frac{n^2 - 1}{4}.
\]

\textit{Case~(2): $\overline{H}$ is bipartite.} Fix a bipartition $V(H) = A \sqcup B$ of $\overline{H}$ with $|A| = a \le b = |B|$ and $a + b = n - 2$. Let
\begin{equation}\label{eq:r-def}
r \;:=\; ab - |E(\overline{H})|
\end{equation}
denote the number of cross-pairs of $A \times B$ that are \emph{non}-edges of $\overline{H}$, equivalently the number of cross-edges of $H$ between $A$ and $B$. Since $E(H)$ decomposes into the edges of $K_A$, the edges of $K_B$, and exactly $r$ cross-edges between $A$ and $B$,
\begin{equation}\label{eq:EH-bipartite}
|E(H)| \;=\; \binom{a}{2} + \binom{b}{2} + r.
\end{equation}
Three subcases arise on $(a, r)$.

\textit{Sub-case~(2.1): $a = 0$ ($\overline{H}$ edgeless).} Then $H = K_{n - 2}$, and substituting $|E(H)| = \binom{n - 2}{2}$ into~\eqref{eq:caseB-decomp} together with $|X_{uv}| \ge 2$ gives
\[
|E(G)| \;\ge\; \binom{n - 2}{2} + (n - 2) + 2 \;=\; \frac{n^2 - 3n + 6}{2} \;\ge\; \frac{n^2 - 1}{4}
\]
for every $n \ge 7$, since $n^2 - 6n + 13 = (n - 3)^2 + 4 > 0$.

\textit{Sub-case~(2.2): $a \ge 1$ and $r \le a - 1$.} Since $H$ is connected and both parts are nonempty, $r \ge 1$; hence $a \ge 2$ in this sub-case. Since $r < a \le b$, by pigeonhole there exist vertices $p \in A$ and $q \in B$, each incident to no cross-edge of $H$; we call such vertices \emph{cross-isolated}. The pair $\{p, q\}$ is a non-edge of $H$, hence of $G$. A common $H$-neighbour of $p$ and $q$ in $A$ would require a cross-edge incident with $q$; one in $B$ would require a cross-edge incident with $p$. Both are impossible, so $N_H(p)\cap N_H(q)=\emptyset$. Applying the Case~B hypothesis at $\{p, q\}$,
\begin{equation}\label{eq:CB-good-pair}
|N_G(p) \cap N_G(q)| \;\ge\; 2,
\end{equation}
and since no vertex of $V(H)$ contributes to this intersection, both $u$ and $v$ must lie in it:
\begin{equation}\label{eq:uv-in-Nab}
\{u, v\} \;\subseteq\; N_G(p) \cap N_G(q).
\end{equation}
The conclusion~\eqref{eq:uv-in-Nab} holds for every pair $(p, q) \in A \times B$ of cross-isolated vertices. Since each side contains at least one cross-isolated vertex, fixing a witness in $B$ and varying $p \in A$ yields $\{u, v\} \subseteq N_G(p)$ for every cross-isolated $p \in A$. By symmetry, $\{u, v\} \subseteq N_G(q)$ for every cross-isolated $q \in B$. The set of cross-isolated vertices of $A$ has size at least $a - r$, and that of $B$ at least $b - r$; hence
\begin{equation}\label{eq:coupling-bound}
|X_{uv}| \;=\; |N_G(u) \cap N_G(v)| \;\ge\; (a - r) + (b - r) \;=\; (n - 2) - 2r.
\end{equation}
Combining~\eqref{eq:caseB-decomp},~\eqref{eq:EH-bipartite}, and~\eqref{eq:coupling-bound},
\begin{equation}\label{eq:case2-2-pre}
\begin{aligned}
|E(G)| &\;\ge\; \binom{a}{2} + \binom{b}{2} + r + (n - 2) + (n - 2) - 2r \\
       &\;=\; \binom{a}{2} + \binom{b}{2} + 2(n - 2) - r.
\end{aligned}
\end{equation}
The right-hand side of~\eqref{eq:case2-2-pre} is bounded below as follows. By convexity of $\binom{x}{2}$, the sum $\binom{a}{2} + \binom{b}{2}$ is minimised at the balanced integer split $a = (n - 3)/2$, $b = (n - 1)/2$, with value $(n - 3)^2/4$. Moreover, $r \le a - 1 \le (n - 5)/2$ since $a \le (n - 3)/2$. Combining,
\begin{equation}\label{eq:case2-2-final}
|E(G)| \;\ge\; \frac{(n - 3)^2}{4} + 2(n - 2) - \frac{n - 5}{2} \;=\; \frac{n^2 + 3}{4} \;\ge\; \frac{n^2 - 1}{4}.
\end{equation}

\textit{Sub-case~(2.3): $a \ge 1$ and $r \ge a$.} By~\eqref{eq:caseB-decomp},~\eqref{eq:EH-bipartite}, and $|X_{uv}| \ge 2$,
\begin{equation}\label{eq:case2-3-pre}
|E(G)| \;\ge\; \binom{a}{2} + \binom{b}{2} + r + (n - 2) + 2 \;\ge\; \binom{a}{2} + \binom{b}{2} + a + n.
\end{equation}
$\binom{a}{2} + \binom{n - 2 - a}{2} + a + n$ is a quadratic in $a$ minimised at $a = (n - 3)/2$. Since $(n - 3)/2$ is integer ($n$ odd) and coincides with the right endpoint of the domain $\{1, \ldots, (n - 3)/2\}$, the right-hand side is non-increasing on the domain, with minimum at $a = (n - 3)/2$. Substituting,
\begin{equation}\label{eq:case2-3-final}
|E(G)| \;\ge\; \frac{(n - 3)^2}{4} + \frac{n - 3}{2} + n \;=\; \frac{n^2 + 3}{4} \;\ge\; \frac{n^2 - 1}{4}.
\end{equation}

In all cases $|E(G)| \ge (n^2 - 1)/4$. By~\eqref{eq:3D-EG},
\[
\rho_{\Delta}(G) \;\ge\; \left\lceil \frac{|E(G)|}{3} \right\rceil \;\ge\; \left\lceil \frac{n^2 - 1}{12} \right\rceil.
\]
It remains to verify that $\lceil (n^2 - 1)/12 \rceil \ge \phi_n$ for every odd $n \ge 7$. By~\eqref{eq:phi-odd}, this is equivalent to $n^2 - 1 > n^2 - n + b_n - 12$, i.e., $n > b_n - 11$, which holds since $b_n \le 12$.
\end{proof}

\begin{theorem}\label{thm:lower-bound}
For every odd $n \ge 7$ and every triangular graph $G$ of order $n$ with $\alpha(G) \le 2$,
\[
\rho_{\Delta}(G) \;\ge\; \phi_n.
\]
\end{theorem}

\begin{proof}
If $\alpha(G) = 1$, then $G = K_n$ and $\rho_{\Delta}(G) = |C_n| \ge \phi_n$ by~\Cref{lem:clique-case}.

If $\alpha(G) = 2$, the dichotomy of Case~A / Case~B applies, and~\Cref{lem:case-a-general} (for Case~A) or~\Cref{lem:case-b-general} (for Case~B) gives $\rho_{\Delta}(G) \ge \phi_n$.
\end{proof}

\begin{proof}[Proof of~\Cref{thm:main}]
By~\Cref{prop:bijection}, $\gamma(J(n, 3))$ equals the minimum of $\rho_{\Delta}(G)$ over triangular graphs of order $n$ with $\alpha(G) \le 2$. The upper bound $\le \phi_n$ is~\Cref{prop:phi-upper}; the lower bound $\ge \phi_n$ is~\Cref{thm:lower-bound}. Hence $\gamma(J(n, 3)) = \phi_n$.
\end{proof}

\section*{Concluding remarks}

Combining~\Cref{thm:main} with~\Cref{thm:CDT-even} gives the identity $\gamma(J(n, 3)) = L(n, 3, 3, 2) = \phi_n$ for every $n \ge 6$ due to Bate~\cite{Bate1978} and Brouwer~\cite{Brouwer1981}, with $\phi_n$ given by the residue formulas~\eqref{eq:phi-even} (for even $n$) and~\eqref{eq:phi-odd} (for odd $n$), derived from the Fort\textendash{}Hedlund covering numbers~\cite{FortHedlund1958}.

The proof of~\Cref{thm:CDT-even} in~\cite[Theorem~11]{CornetDravecTorres2025} considers a graph satisfying $\rho_\Delta(G)=t_n$ and infers that if $G-\{u,v\}$ is disconnected for a non-adjacent pair $u,v$, then $V(G)$ can be partitioned into two cliques. The structural assumptions of triangularity and $\alpha(G)=2$ alone do not justify this inference. Take non-adjacent vertices $u,v$, a clique $C_1$ with $|C_1|\ge2$ whose vertices are adjacent to both $u$ and $v$, and a clique $\{y_1,y_2,z_1,z_2\}$ with no edges to $C_1$, where $v$ is adjacent to $y_1,y_2$ and $u$ is adjacent to $z_1,z_2$. This graph is triangular, has $\alpha(G)=2$, and $G-\{u,v\}$ has two components, but $V(G)$ admits no partition into two cliques. This example is not asserted to satisfy the additional optimality assumption $\rho_\Delta(G)=t_n$. A justification of the step could use that assumption, or the weaker structural reduction in~\Cref{lem:disconnection-2clique} together with an even-order version of Lemma~10 of~\cite{CornetDravecTorres2025}. We prove the latter lemma here only for odd $n$. The statement of~\Cref{thm:CDT-even} follows from Bate and Brouwer, and Bate's short argument recalled in the introduction handles both parities.

For $J(n,k)$ with $k\ge4$, equivalently for the lotto numbers $L(n,k,k,k-1)$, we are not aware of an exact determination of $\gamma$ for any fixed $k\ge4$ and all $n$; this question is also posed in~\cite[Problem~29]{CornetDravecTorres2026}. Proposition~15 of~\cite{CornetDravecTorres2025} sketches a $(k-1)$-uniform hypergraph framework generalising~\Cref{prop:bijection}.

A natural follow-up to~\Cref{thm:main} is the extremal characterisation of those $2$-shadow graphs $G$ at which the minimum in~\Cref{prop:bijection} is attained. The family $K_{m_1}\sqcup K_{m_2}$ from~\Cref{prop:phi-upper}, with a split attaining $\phi_n$, provides one such class. Extremal shadows need not be disjoint unions of cliques: for $n=7$, the family
\[
D=\bigl\{\{1,2,3\},\{4,5,7\},\{4,6,7\},\{1,5,6\}\bigr\}
\]
has a shadow consisting of the clique on $\{1,2,3\}$, the clique on $\{4,5,6,7\}$, and the two additional edges $\{1,5\},\{1,6\}$. This shadow is connected with $\alpha=2$, so $D$ dominates by~\Cref{prop:bijection}, and $|D|=4=\phi_7$.

In our proof, Case~A is resolved by~\Cref{lem:CDT9} and~\Cref{lem:CDT10}, without classifying equality. Case~B gives $|E(G)|\ge(n^2-1)/4$, and hence $\rho_\Delta(G)\ge\lceil(n^2-1)/12\rceil$. This numerical lower bound equals $\phi_n$ precisely for $n\in\{7,9,11\}$. For odd $n\ge13$ it exceeds $\phi_n$: for $n\ge15$ this follows from $n-1>b_n$, while at $n=13$ the two values are $14$ and $13$. Thus no Case~B graph attains the minimum for odd $n\ge13$, and every extremal $G$ lies in Case~A. For the same $n$, a different structural statement follows from~\cite[Theorem~2]{Brouwer1981}: an extremal shadow has at most $3\phi_n\le m^2+m-2$ edges, where $m=\lfloor n/2\rfloor$, so its vertex set is the union of two cliques. At $n\in\{7,9,11\}$, neither argument applies: the Case~B bound does not exclude Case~B extremals, and $3\phi_n>m^2+m-2$. The degree argument of~\cite{FurediSzekelyZubor1996} does not preserve the shadow, so it gives no structural information here either. A full classification of extremal shadows is left to future work.

For $r\ge2$ and $n\ge2r$, the nondecreasing property $\gamma(J(n,r))\le\gamma(J(n+1,r))$ was proved in~\cite[Theorem~4]{CornetDravecTorres2025} and is restated in~\cite[Theorem~4]{CornetDravecTorres2026}. The latter paper also records the two-step increase $\gamma(J(n,r))+1\le\gamma(J(n+2,r))$ under the same hypotheses~\cite[Theorem~5]{CornetDravecTorres2026}. Conjecture~14 of~\cite{CornetDravecTorres2025} asserts that $\gamma(J(n,r))<\gamma(J(n+1,r))$ for every $r\ge3$ and every $n\ge2r$. The case $r=3$ is Theorem~13 of~\cite{CornetDravecTorres2025}; it can also be read off from~\eqref{eq:phi-even} and~\eqref{eq:phi-odd}, which give $\phi_{n+1}-\phi_n\ge1$ for every $n\ge6$. The cases $r\ge4$ remain open.

\section*{Acknowledgments}

The work of Semin Oh was supported by Global \textendash{} Learning \& Academic research institution for Master's $\cdot$ PhD students, and Postdocs (G-LAMP) Program of the National Research Foundation of Korea (NRF) grant funded by the Ministry of Education (No.~RS-2023-00301914).

\bibliographystyle{plain}
\bibliography{refs}

\end{document}